\documentclass{amsart}
\usepackage{graphicx} % Required for inserting images
\usepackage{xfrac}
\usepackage{amsthm}
\usepackage{amsmath}
\usepackage{latexsym, amssymb}
\usepackage{mathtools}
\usepackage{color}
\usepackage{xcolor}
\usepackage{comment}
\usepackage{tikz-cd}
\usepackage[hidelinks]{hyperref}
\usepackage[all]{xy}

\newtheorem{thm}{Theorem}[section] %the resolution could also be [subsection]

\newtheorem{cor}[thm]{Corollary}
\newtheorem{defn}[thm]{Definition}

\newtheorem{lem}[thm]{Lemma}

\theoremstyle{definition}

\newtheorem*{thm*}{Theorem}

\newcommand\operA[2]{{\if!#2!\operatorname{#1}\else{\operatorname{#1}_{#2}^{\phantom{I}}}\fi}} % To be used within Bdefs. Usage: $\operA{N}{K/F}$ produces $N_{K/F}$; $\operA{N}{}$ produces $N$.
\newcommand\set[1]{\{#1\}}

\newcommand\Cref[1]{{Corollary~\ref{#1}}}%
\newcommand{\Trace}[1][]{\if!#1!\operatorname{Tr}\else{\operatorname{Tr}_{#1}^{\phantom{I}}}\fi} % Usage: $\Tr[K/F](a)$.

\long\def\forget#1\forgotten{{}} %
\newcommand\suchthat{{\,:\ \,}}

\def\({\left(}
\def\){\right)}

\newcommand\LAY[3][]{{\begin{array}{c}\mbox{#2} \if#1!{}\else{+}\fi \\ \mbox{#3}\end{array}}}

\makeatletter
\DeclareRobustCommand{\abs}{\@ifstar\star@abs\normal@abs}
\newcommand{\normal@abs}[2][]{\mathopen{#1|}#2\mathclose{#1|}}
\makeatother

\makeatletter

\def\ps@pprintTitle{%
 \let\@oddhead\@empty
 \let\@evenhead\@empty
 \def\@oddfoot{}%
 \let\@evenfoot\@oddfoot}

\newcommand{\bigperp}{%
  \mathop{\mathpalette\bigp@rp\relax}%
  \displaylimits
}

\newcommand{\bigp@rp}[2]{%
  \vcenter{
    \m@th\hbox{\scalebox{\ifx#1\displaystyle2.1\else1.5\fi}{$#1\perp$}}
  }%
}
\makeatother

\renewcommand{\geq}{\geqslant}
\renewcommand{\leq}{\leqslant}

\newif\iffurther
\furtherfalse
\title{How Associative Can a Non-Associative Moufang Loop Be?}
\author{Ilan Levin}
\address{Department of Mathematics, Bar-Ilan University, Ramat Gan, 55200, Israel}
\email{ilan7362@gmail.com}
\date{}
\subjclass{
{Primary: 20N05; Secondary:
05B07; 05E16; 60B99
}    }
\begin{document}

\begin{abstract}
    We prove a non-associative analog to the well-known $\frac{5}{8}$ Theorem. Namely, for a finite Moufang loop with nuclear commutators, we show that if the probability that three randomly chosen elements associate is greater than $\frac{43}{64}$, then the loop must be a group. The bound is tight as demonstrated by the 16-element Octonion loop.
\end{abstract}

\maketitle

The $\frac{5}{8}$ Theorem is a well-known observation in group theory; given a finite group $G$, if the probability of two randomly chosen elements in $G$ commuting exceeds $\frac{5}{8}$, then the group is commutative. This probability is usually denoted by $\mathbb{P}(G)$. The bound is sharp, as the probability of two randomly chosen elements in the Quaternion group commuting is indeed $\frac{5}{8}$. The result was proven and generalized to compact groups with Haar measure in \cite{Gustaf}.
This intriguing property has been studied in multiple different directions. In \cite{Rusin}, all groups with $\mathbb{P}(G) > \frac{11}{32}$ are completely classifed. In particular, it is shown that if $\mathbb{P}(G) > \frac{11}{32}$, then  $\mathbb{P}(G) \in \set{\frac{1}{2},\frac{7}{16},\frac{11}{27},\frac{2}{5},\frac{25}{64},\frac{3}{8}} \cup \set{\frac{1}{2} + \frac{1}{2^{2s+1}} \suchthat s \in \mathbb{N}}$. 
In \cite{Solvable}, it is shown that if two randomly chosen elements of a finite group $G$ generate a solvable group with probability exceeding $\frac{11}{30}$, then the whole group is solvable. Moreover, the bound is attained by the alternating group $A_5$.

\noindent This short note provides two non-associative analogs to the $\frac{5}{8}$ Theorem:
\begin{thm*}
    Let $G$ be a finite non-associative Moufang loop satisfying \break $[[G, G], G, G] = 1$. Then the probability that three uniformly chosen elements in $G$ associate does not exceed $\frac{43}{64}$.
\end{thm*}
\begin{thm*}
    Let $G$ be a finite non-associative CC loop. Then the probability that three uniformly chosen elements in $G$ associate does not exceed $\frac{7}{8}$.
\end{thm*}

The bounds given in the theorems are sharp. They are attained when considering the octonion loop $\mathbb{O}_{16}$ and the smallest CC loop of order 6, respectively.
\begin{comment}
\begin{thm*}
    Let $G$ be a finite non-associative CC loop satisfying. Then the probability that three uniformly chosen elements in $G$ associate does not exceed $\frac{7}{8}$.
\end{thm*}
\end{comment}

\section{Preliminaries}

A set $L$ with a binary operation is called a {\bf{quasigroup}} if every equation of the form $ax = b$ or $xa = b$ has a unique solution ($x$ being the variable). 
A quasigroup with an identity element $1$ is called a loop.
A loop $G$ is called Moufang if one of the following equivalent identities holds:
\begin{align}
    z(x(zy)) = ((zx)z)y \\
    ((xz)y)z = x(z(yz)) \\
    (zx)(yz) = (z(xy))z 
\end{align}

The associator of $a,b,c$ is denoted $[a,b,c]$ and is defined as the unique solution to the equation $(ab)c = (a(bc))x$. \begin{defn}
    The left (middle, right) nucleus of $G$ is defined as the set of all $a$'s such that $[a,x,y] = 1$ ($[x,a,y] = 1$, $[x,y,a] = 1$ respectively) for any $x,y \in G$. It is denoted by $\mathfrak{N}_{\ell}(G)$ ($\mathfrak{N}_{m}(G)$, $\mathfrak{N}_{r}(G)$ respectively).
\end{defn}
In Moufang loops, all the nuclei coincide. We denote the nucleus by $\mathfrak{N}(G)$, and when $G$ is understood from context, we even use $\mathfrak{N}$. The nucleus is a normal subloop.

Moufang loops satisfy the fundamental Moufang's theorem:
\begin{thm} \label{Moufangthm}
    Let $G$ be a Moufang loop and let $x,y,z \in G$ be elements such that $[x,y,z] = 1$. Then $\langle x,y,z \rangle$ is an associative subloop. In particular, whenever the associator is trivial, it is invariant under permutations and inverses of the elements.
\end{thm}

%For the original proof see \cite{Moufang1935}; for a simplified proof see \cite{drapalsimp}. \\
Moufang loops are di-associative, that is, every subloop generated by two elements is associative. This fact follows immediately from Theorem \ref{Moufangthm}. In particular, this fact implies that Moufang loops have well-defined inverses, and that they satisfy the alternative laws: $(xy)y = x(yy)$ and $(xx)y = x(xy)$. \qquad

A finite loop is said to have the Lagrange property if every subloop is of order dividing the order of the loop. Moufang loops satisfy the Lagrange property, as proven in \cite{Lagrange}.

We mention Bruck's notion of an adjoint set, appearing in \cite[Page 119]{Bruck1971ASO}.
\begin{defn}
    Given a subset $A \subseteq G$, the adjoint of $A$, $A'$, is the set of all $x \in G$ such that $[A, x, G] = 1$.
\end{defn}

By \cite[Lemma VII.4.3]{Bruck1971ASO}, if $G$ is a Moufang loop, $A'$ is a subloop.

\section{Moufang Loops with Nuclear Commutators}
Throughout this section, $G$ is a Moufang loop and $\mathfrak{N}$ is its nucleus.
We start by realizing how big the nucleus can be.
If $G$ is associative, then of course $\mathfrak{N} = G$. Let us assume otherwise.

\begin{lem}
    Let $G$ be a Moufang loop and $A \leq \mathfrak{N}$, such that $G/A$ is generated by at most two elements. Then $G$ is a group.
\end{lem}
\begin{proof}
    Assume $G/A = \langle xA, yA \rangle$. Denote $H = \langle x, y \rangle$. Notice that $HA/A$ contains both $xA$ and $yA$, so $HA/A = G/A$ and $HA = G$. Hence, the loop can be written as $G = \langle A, x, y \rangle$. By \cite[Theorem VII.4.2]{Bruck1971ASO}, taking $A = A, B = \set{x}, C = \set{y}$, we see that $A \cup \set{x,y}$ is contained in an associative subloop of $G$. But this can only be $G$ - so $G$ is associative.
\end{proof}

\begin{cor} \label{8lem}
    A non-associative Moufang loop $G$ satisfies $[G: \mathfrak{N}] \geq 8$.
\end{cor}
\begin{proof}
    This follows by the previous lemma and the classification of groups up to order 7.
\end{proof}

This bound cannot be improved, as the 16-element octonion loop $\mathbb{O}_{16}$ (studied extensively in \cite{kirshtein}) is Moufang and $\mathfrak{N}(\mathbb{O}_{16}) = \{ \pm 1 \}$. In this case, the quotient is isomorphic to $\mathbb{Z}/2\mathbb{Z} \times \mathbb{Z}/2\mathbb{Z} \times \mathbb{Z}/2\mathbb{Z}$.

For any $x \in G$, define the left translation map $L_x : G \to G$
and the right translation map $R_x : G \to G$ by $L_x(y) = xy$ and $R_x(y) = yx$, respectively.
%For any $x \in G$, define the left translation map $L_x : G \to G$ be the left translation map, $y \mapsto xy$, and similarly, $R_x : G \to G$ be the right translation map, $y \mapsto yx$.
Denote by $L(x,y)$ the composition $L_x \circ L_y \circ L_{yx}^{-1}$. This map is a pseudo-automorphism with companion $[y,x]$ for all $x,y \in G$. Pseudo-automorphisms are introduced in \cite[Page 113]{Bruck1971ASO}.
\begin{lem} \label{2.2}
    If $[G,G,[G,G]] = 1$, then the set $\partial_{x,y} := \set{z \suchthat [z,y,x] = 1}$ is a subloop for all $x,y \in G$.
\end{lem}
\begin{proof}
    By assumption, $[y,x] \in \mathfrak{N}$, so $L(x,y)$ is an automorphism. By \cite[Lemma VII.5.4]{Bruck1971ASO}, $L(x,y)(z) = z[z,y,x]^{-1}$. The set of fixed points of an automorphism forms a subloop, and $L(x,y)(z) = z$ if and only if $[z,y,x] = 1$.
\end{proof}

We are finally ready to prove the main theorem.

\begin{thm}
    Let $G$ be a finite non-associative Moufang loop that satisfies $[G, G, [G, G]] = 1$. Then the probability that three elements $x,y,z$ chosen uniformly associate is at most $\frac{43}{64}$.
\end{thm}
\begin{proof}
    First of all, notice that the event $x,y,z$ associate can be split into three disjoint cases:
    \begin{enumerate}
        \item $x \in \mathfrak{N}$;
        \item $x \notin \mathfrak{N}$ and $y \in \set{x}'$;
        \item $x \notin \mathfrak{N}$ and $y \notin \set{x}'$ and $z \in \partial_{x,y}$.
    \end{enumerate}
    The probability of the first case is $\frac{\abs{\mathfrak{N}}}{\abs{G}}$. Assume $x \notin \mathfrak{\mathfrak{N}}$. So there exist $u,v \in G$ such that $[x,u,v] \neq 1$. Thus, the chain of subloops $$\mathfrak{N} \leq \set{x}' \leq \partial_{x,u} \leq G$$ is strictly ascending. A well-known exercise in quasigroup theory states that a proper subquasigroup is of size at most half of the quasigroup. In particular, this means that $\abs{\set{x}'} \leq \frac{\abs{G}}{4}$. So the probability of the second case is bounded by $(1-\frac{\abs{\mathfrak{N}}}{\abs{G}}) \cdot \frac{1}{4}$. The probability of the third case is bounded by $(1-\frac{\abs{\mathfrak{N}}}{\abs{G}}) \cdot \frac{3}{4} \cdot \frac{1}{2} = (1-\frac{\abs{\mathfrak{N}}}{\abs{G}}) \cdot \frac{3}{8}$, since there exists $v \in G$ such that $[x,y,v] \neq 1$, so $v \notin \partial_{x,y}$. Summing everything up, the probability of $x,y,z$ associating is at most $\frac{\abs{\mathfrak{N}}}{\abs{G}} \cdot \frac{3}{8} + \frac{5}{8}$. By Corollary~\ref{8lem}, $\frac{\abs{\mathfrak{N}}}{\abs{G}} \leq \frac{1}{8}$, so the best case scenario is when $\frac{\abs{\mathfrak{N}}}{\abs{G}} = \frac{1}{8}$, and we get that the probability is bounded by $\frac{43}{64}$. 
\end{proof}

\begin{cor}
    If the probability that $3$ elements chosen uniformly associate in a Moufang loop $G$ with nuclear commutators is strictly greater than $\frac{43}{64}$, then $G$ is a group.
\end{cor}

The bound is sharp. It is attained when considering the octonion loop $\mathbb{O}_{16}$.

\section{Conjugacy Closed Loops}

A loop $G$ is called a conjugacy closed loop (or CC loop) if it satisfies the following two identities:

\begin{align}
    z(xy) = R_z^{-1}(zx) \cdot (zy) \\
    (xy)z = (xz) \cdot L_z^{-1}(yz) 
\end{align}

Conjugacy closed loops can also be interpreted as loops in which left translation maps and right translation maps are closed under conjugation, that is, for all $x,y \in G$, $R_x^{-1} \circ R_y \circ R_x$ is a right translation map and $L_x^{-1} \circ L_y \circ L_x$ is a left translation map. 

Associators of CC loops are invariant under permutations. Moreover, CC loops satisfy the Lagrange property; see \cite[Corollary 3.2]{CCcite}.

We prove a similar theorem for CC loops.
\begin{thm}
    Let $G$ be a finite non-associative CC loop. Then the probability that three elements $x,y,z$ chosen uniformly associate is at most $\frac{7}{8}$.
\end{thm}
\begin{proof}
    Let $G$ be a finite non-associative CC loop.
    Similarly to the proof for Moufang loops with nuclear commutators, the event $x,y,z$ associate can be split into three disjoint cases:
    \begin{enumerate}
        \item $x \in \mathfrak{N}$;
        \item $x \notin \mathfrak{N}$ and $y \in \set{x}'$;
        \item $x \notin \mathfrak{N}$ and $y \notin \set{x}'$ and $z \in \partial_{x,y}$.
    \end{enumerate}
     The probability of the first case is $\frac{\abs{\mathfrak{N}}}{\abs{G}}$. Now assume $x \notin \mathfrak{N}$. So, there exist $u,v \in G$ such that $[x,u,v] \neq 1$. Using the same notation as in Lemma \ref{2.2}, the set $\partial_{x,u}$ is a subloop of $G$ by \cite[Lemma 2.9(5)]{PACC}, but $v \notin \partial_{x,u}$, so $\abs{\partial_{x,u}} \leq \frac{\abs{G}}{2}$. The set $\set{x}'$ is contained in $\partial_{x,u}$, hence $\abs{\set{x}'} \leq \frac{\abs{G}}{2}$ as well. Taking this fact into account, we find that the probability of the second case is bounded by $(1-\frac{\abs{\mathfrak{N}}}{\abs{G}}) \cdot \frac{1}{2}$, and the probability of the third case is bounded by $(1-\frac{\abs{\mathfrak{N}}}{\abs{G}}) \cdot \frac{1}{2} \cdot \frac{1}{2}$. In summary, the probability of the associating $x,y,z$ is at most $\frac{\abs{\mathfrak{N}}}{\abs{G}} \cdot \frac{1}{4} +  \frac{3}{4}$. But since $\mathfrak{N} \neq G$, we get $\abs{\mathfrak{N}} \leq \frac{\abs{G}}{2}$. So, the probability is bounded by $\frac{1}{2} \cdot \frac{1}{4} + \frac{3}{4} = \frac{7}{8}$.
\end{proof}

\begin{cor}
    If the probability that $3$ elements chosen uniformly associate in a CC loop $G$ is strictly greater than $\frac{7}{8}$, then $G$ is a group.
\end{cor}

This bound is sharp. It is attained when considering the smallest CC loop, which is of order 6.

\section{Concluding Remarks}

\begin{enumerate}
    \item The author suspects that the theorem might be false for general Moufang loops, due to the concluding remark in \cite{suspicion}.
    \item Verifications of the probability for the loop $\mathbb{O}_{16}$ as well as the smallest CC loop of order 6 were made using the LOOPS package \cite{loops} in GAP.
\end{enumerate}

\section*{Acknowledgement}

The author thanks Uzi Vishne for the helpful discussions and suggestions.

\section*{Disclosure Statement}

The author reports there are no competing interests to declare.

\bibliographystyle{abbrv}
\bibliography{final}

\def\cprime{$'$}
\begin{thebibliography}{10}

\bibitem{suspicion}
R.~H. Bruck.
\newblock Pseudo-automorphisms and moufang loops.
\newblock {\em Proceedings of the American Mathematical Society}, 3(1):66--72, 1952.

\bibitem{Bruck1971ASO}
R.~H. Bruck.
\newblock A survey of binary systems.
\newblock In {\em Ergebnisse der Mathematik und Ihrer Grenzgebiete}, 1971.

\bibitem{Lagrange}
A.~Grishkov and A.~Zavarnitsine.
\newblock Lagrange's theorem for moufang loops.
\newblock {\em Mathematical Proceedings of the Cambridge Philosophical Society}, 139:41 -- 57, 07 2005.

\bibitem{Solvable}
R.~M. Guralnick and J.~S. Wilson.
\newblock The probability of generating a finite soluble group.
\newblock {\em Proceedings of the London Mathematical Society}, 81(2):405--427, 2000.

\bibitem{Gustaf}
W.~H. Gustafson.
\newblock What is the probability that two group elements commute?
\newblock {\em The American Mathematical Monthly}, 80(9):1031--1034, 1973.

\bibitem{PACC}
M.~K. Kinyon and K.~Kunen.
\newblock Power-associative, conjugacy closed loops.
\newblock {\em Journal of Algebra}, 304(2):679--711, 2006.

\bibitem{CCcite}
M.~K. Kinyon, K.~Kunen, and J.~D. Phillips.
\newblock Diassociativity in conjugacy closed loops.
\newblock {\em Communications in Algebra}, 32(2):767--786, 2004.

\bibitem{kirshtein}
J.~Kirshtein.
\newblock Automorphism groups of cayley-dickson loops.
\newblock {\em Journal of Generalized Lie Theory and Applications}, 6, 2012.

\bibitem{loops}
G.~P. Nagy and P.~Vojt{\v e}chovsk{\a'y}.
\newblock {loops}, computing with quasigroups and loops in gap, {V}ersion 3.4.4.
\newblock \href {https://gap-packages.github.io/loops/} {\texttt{https://gap\texttt{\symbol{45}}packages.github.io/}\discretionary {}{}{}\texttt{loops/}}, Aug 2024.
\newblock GAP package.

\bibitem{Rusin}
D.~J. Rusin.
\newblock {What is the probability that two elements of a finite group commute?}
\newblock {\em Pacific Journal of Mathematics}, 82(1):237 -- 247, 1979.

\end{thebibliography}
\end{document}